\documentclass{amsart}
\RequirePackage{lmodern}
\RequirePackage[T1]{fontenc}
\RequirePackage[final]{microtype}
\RequirePackage{xcolor}
\selectcolormodel{rgb}

\RequirePackage[leqno]{amsmath}
\RequirePackage{amssymb,amsopn,amsthm,bm,mathrsfs}
\RequirePackage{mathtools}

\RequirePackage{cite,indentfirst}
\RequirePackage{enumerate,xspace,centernot,paralist}
\RequirePackage{etoolbox}
\RequirePackage{suffix}

\makeatletter
  \@ifclassloaded{amsart}{%
    \patchcmd{\section}{\scshape\centering}{\bfseries\large}{}{}
  }
  {}
\makeatother

\RequirePackage{hyperref}

\definecolor{link1}{rgb}{0,0,.7}
\definecolor{link2}{rgb}{0,0.25,0.5}
\makeatletter
\@ifpackagelater{hyperref}{2012/05/28}{
  \hypersetup{final,colorlinks,allcolors=link1,citecolor=link2}
}{
  \@ifpackagelater{hyperref}{2009/10/09}{
    \hypersetup{final,colorlinks,linkcolor=link1,anchorcolor=link1,filecolor=link1,menucolor=link1,runcolor=link1,urlcolor=link1,citecolor=link2,linktocpage=false,breaklinks=true}
  }{
    \hypersetup{final}
  }
}
\makeatother

\allowdisplaybreaks

\numberwithin{equation}{section}

\theoremstyle{plain}
\newtheorem{theorem}{Theorem}[section]
\newtheorem{lemma}[theorem]{Lemma}
\newtheorem{proposition}[theorem]{Proposition}

\newtheorem{conjecture}[theorem]{Conjecture}

\newtheorem*{theorem*}{Theorem}
\newtheorem*{lemma*}{Lemma}
\newtheorem*{proposition*}{Proposition}
\newtheorem*{corollary*}{Corollary}

\theoremstyle{definition}

\theoremstyle{remark}
\newtheorem{remark}[theorem]{Remark}
\newtheorem*{remark*}{Remark}

\newtheoremstyle{misc}
  {}
  {}
  {\itshape}
  {}
  {\bfseries}
  {.}
  { }
  {\thmname{#3}\thmnumber{ #2}}
\theoremstyle{misc}

\newtheorem*{miscthm*}{Misc}
\theoremstyle{plain}

\providecommand{\cites}[1]{\cite{#1}}

\newcounter{GICase}
\newenvironment{proofcases}{\setcounter{GICase}{0}}{}
\newcommand{\case}[1]{\smallskip\par\noindent\stepcounter{GICase}\emph{Case \Roman{GICase}: #1.}\xspace}

\newcounter{GIStep}

\newcommand{\dv}{\nabla \cdot}
\newcommand{\curl}{\nabla \times}

\newcommand{\prob}{\bm{P}}

\newcommand{\eps}{\varepsilon}

\newcommand{\R}{\mathbb{R}}
\newcommand{\C}{\mathbb{C}}
\newcommand{\N}{\mathbb{N}}

\newcommand{\FixDelimSubscripts}[1]{%
  \reDeclarePairedDelimiterInnerWrapper{#1}{nostar}{##1##2##3}
}
\newcommand{\newdelim}[3]{%
  \DeclarePairedDelimiter{#1}{#2}{#3}
  \FixDelimSubscripts{#1}
}
\newdelim{\paren}{(}{)}
\newdelim{\brak}{[}{]}
\newdelim{\set}{\{}{\}}
\newdelim{\floor}{\lfloor}{\rfloor}
\newdelim{\ceil}{\lceil}{\rceil}

\newdelim{\abs}{\lvert}{\rvert}
\newdelim{\qv}{\langle}{\rangle}
\newdelim{\av}{\langle}{\rangle}

\newif\ifdisplaystyle\displaystylefalse
\everydisplay\expandafter{\the\everydisplay\displaystyletrue}
\DeclarePairedDelimiter{\norm}{\lVert}{\rVert}
\reDeclarePairedDelimiterInnerWrapper{\norm}{nostar}{%
  \ifdisplaystyle\mathopen{#1}#2\mathclose{#3}%
  \else#1#2#3\fi%
}

\newcommand{\given}[1][\big]{\mathchoice{\;#1|\;}{\:|\:}{\:|\:}{\:|\:}}

\newcommand{\defeq}{\stackrel{\scriptscriptstyle\textup{def}}{=}}

\usepackage[foot]{amsaddr}

\newtheorem*{heuristic}{Heuristic ODE}

\newcommand{\sfrac}[2]{#1/#2}
\newcommand{\GF}{G_{\!F,\eps}}
\newcommand{\nablaA}{\nabla_{\!\!A}}
\newcommand{\nablaAe}{\nabla_{\!\!A_\eps}}
\newcommand{\nablaAep}{\nabla_{\!\!A_\eps'}}
\newcommand{\Ee}{\mathcal E_\eps}

\newcommand{\dist}{d}

\newcommand{\hex}{h_{\textit{ex}}}

\newcommand{\logep}{\abs{\ln \eps}}

\begin{document}
  \title[Vortex Nucleation]{A model for vortex nucleation in the Ginzburg-Landau equations}

  \author[Iyer]{Gautam Iyer\textsuperscript{1}}
  \address{%
    \textsuperscript{1}
    Dept.\ of Mathematical Sciences,
    Carnegie Mellon University,
    Pittsburgh, PA 15213.}
  \email{gautam@math.cmu.edu}

  \author[Spirn]{Daniel Spirn\textsuperscript{2}}
  \address{%
    \textsuperscript{2}
    School of Mathematics,
    University of Minnesota,
    Minneapolis, MN  55455.}
  \email{spirn@math.umn.edu}

  \thanks{This material is based upon work partially supported by the National Science Foundation (through grants
    DMS-1252912 to GI,
    and
    DMS-0955687,
    DMS-1516565
    to DS),
    the Simons Foundation (through grant \#393685 to GI),
    the Center for Nonlinear Analysis (through grant NSF OISE-0967140),
    and the Institute for Mathematics and Applications (IMA)%
  }

  \begin{abstract}
    This paper studies questions related to the dynamic  transition between local and global minimizers in the Ginzburg-Landau theory of superconductivity.
    We derive a heuristic equation governing the dynamics of vortices that are close to the boundary, and of dipoles with small inter vortex separation.
    We consider a small random perturbation of this equation, and study the asymptotic regime under which vortices nucleate.
  \end{abstract}
\maketitle
\section{Introduction.}\label{s:Intro}

This paper studies questions related to the dynamic  transition between local and global minimizers in the Ginzburg-Landau theory of superconductivity.
The Ginzburg-Landau theory provides a mesoscopic description of the state of a superconductor through the \emph{order parameter} -- a specific function $\C$-valued function $u \in H^1(\Omega)$ for which the local density of superconducting Cooper pairs is given by $|u(x)|$. 
Here $\Omega \subseteq \R^2$ is the region occupied by the superconductor.
A fundamental feature of superconductors are the presence of localized regions called \emph{vortices}, where the superconductor drops into a normal state.
In these regions  the degree of $u$ is nontrivial
about each vortex, and the induced magnetic field pierces through the superconductor.   
 
The mechanism by which vortices become energetically favorable was proved by Serfaty using a careful energy decomposition.
Recall, the Ginzburg-Landau  energy is defined by
\begin{equation} \label{e:GLenergy}
G_\eps(u,A) \defeq  \int_\Omega \paren[\Big]{
  \frac{1}{2} \abs{ \nablaA u }^2 + \frac{1}{2} \abs{ \curl A - \hex }^2 + \frac{1}{4 \eps^2} {\paren{ 1 - |u|^2 }}^2} \, dx \,,
\end{equation}
where $A$ is the magnetic potential, $\hex = \hex(\eps)$ is the strength of the external magnetic field, and $\nablaA \defeq \nabla  - i A $.     
Physically, $\eps$ is the non-dimensional ratio of the superconductors coherence length to the London penetration depth.

To understand the energy decomposition, define the \emph{Meissner} potential $\xi_{m}$ to be the solution of
\begin{align}\label{e:Meissner}
\begin{split}
- \Delta \xi_m + \xi_m + 1 & = 0  \quad\hbox{ in } \Omega \\
\xi_m & = 0 \quad\hbox{ on } \partial \Omega \,.
\end{split}
\end{align}
When $\hex$ is small enough (explicitly, when $\hex < h_{c_1}$, defined below), the purely superconducting state with no vortices gives a \emph{global} minimizer of the Ginzburg-Landau energy $G_\eps$, see \cite{SSGlobal}.
This state corresponds to $u \equiv 1$ and $A = \hex \nabla^\perp \xi_m$,
and the minimizing energy (called the \emph{Meissner energy}) is given by
\begin{equation} \label{e:MeisEnergy}
  G_m(\hex)
    \defeq G_\eps(1, \hex \nabla^\perp \xi_m)
    = {\hex^2} \int_\Omega \frac{1}{2} \abs{ \nabla \xi_m }^2 + \frac{1}{2} \abs{ \Delta \xi_m - 1 }^2 \, dx \,.
\end{equation}

If there are a finite number of vortices at points $a_j$, with degrees $d_j \in \{\pm 1\}$ respectively which are reasonably separated and away from the boundary, then Serfaty~\cite{SerfatyLocal, SerfatyARMA} shows that Ginzburg-Landau energy can be decomposed as  
\begin{equation} \label{e:decompenSer}
G_\eps(u_\eps, A_\eps) = G_m + \sum_{j=1}^n \paren{ \pi \logep + 2 \pi d_j  \hex \, \xi_m (a_j) } + o_\eps(\logep)\,,
\end{equation}
and the order parameter $u_\eps$ takes the form,
\begin{equation} \label{e:approxuform}
u_\eps (x) \approx \prod_{j=1}^n \rho_\eps(|x -a_j|) \paren[\Big]{ \frac{x - a_j }{ | x - a_j |} }^{d_j}  e^{i\psi^*_\eps} \,,
\end{equation}
where $\rho_\eps(s)$ is the equivariant vortex profile with $\rho_\eps(0) = 0$, $\rho_\eps(s) \to 1$ for $s \gg \eps$ (see Appendix II of \cite{BBH}), $d_j \in \pm 1$, and $\psi^*_\eps$ is a harmonic function that  ensures $\partial_\nu  u_\eps =0$ 
on $\partial \Omega$.

Note that $\xi_m(a_i)$ is always negative, since the maximum principle implies $-1 < \xi_m < 0$ in $\Omega$.
Thus, examining \eqref{e:decompenSer} one sees that vortices with negative degrees are never energetically favorable.
Furthermore, if the applied magnetic field $\hex$ is very large, then a positively oriented vortex can be energetically favorable.
The critical threshold at which this happens is explicitly given by
\begin{equation}
 h_{c_1} \defeq \frac{\logep }{ 2  \max \abs{ \xi_m } }\,,
\end{equation}
and is known as the \emph{first critical field}.
In this case, the optimal location for a single positively oriented, energetically favorable vortex is at the point where $\xi_m$ achieves its minimum and is located in the interior of~$\Omega$.
\medskip

Our main interest in this paper is to study the dynamic transition between the Meissner state and the energetically favorable state with an interior vortex.
We recall that the dynamics of a type-II superconductor are governed by the Gor'kov-\'Eliashberg system~\cite{GorkovEliashberg65}, a coupled system of equations describing the evolution of the order parameter~$u_\eps$ and the electromagnetic field potentials $\Phi_\eps \in H^1(\R^2, \R^1)$, $A_\eps \in H^1(\R^2, \R^2)$.
Explicitly, these equations are
\begin{alignat}{2}
\label{e:GE1}
\span \partial_{\Phi_\eps} u_\eps  = \nablaAe^2 u_\eps + \frac{1}{\eps^2} u_\eps \paren{ 1 - |u_\eps|^2 }
  &\quad& \text{in } \Omega\,,\\
\label{e:GE2}
\span E_\eps = \nabla^\perp h_\eps + {j_{A_\eps}}(u_\eps)
  && \text{in } \Omega\,,\\
\label{e:GE3}
\span \nu \cdot \nablaAe u_\eps  =\nu \cdot E_\eps = 0
  && \text{on } \partial \Omega\,,\\
\label{e:GE4}
\span h_\eps = \hex 
  && \text{on } \partial \Omega\,,
\end{alignat}
where
\begin{equation*} 
\partial_{\Phi_\eps} \defeq \partial_t  - i \Phi_\eps\,,
\quad
E_\eps \defeq \partial_t {A_\eps} - \nabla \Phi_\eps\,,
\quad
h_\eps \defeq \curl {A_\eps}\,,
\quad
j_{A_\eps}(u_\eps) \defeq \paren{ i u_\eps , \nablaAe u_\eps }\,,
\end{equation*}
and $(a, b) \defeq \frac{1}{2} ( a \overline{b} + \overline {a} b)$ is the real part of the complex inner product.  


Now consider the Gor'kov-\'Eliashberg system with initial data $(u_\eps^0, A_\eps^0)$ corresponding to the Meissner state $( 1, \hex \nabla^\perp \xi_m)$.
Since energy satisfies the diffusive identity
\begin{equation} \label{e:enlaw}
G_\eps(u_\eps(t), A_\eps(t))
  + \int_0^t \int_\Omega \paren[\Big]{ \abs{ \partial_{\Phi_\eps} u_\eps  }^2 + \abs{E_\eps}^2} \, dx \, ds = G_\eps(u_\eps^0, A_\eps^0)\,,
\end{equation}  
we will assume that $\hex$ is large enough so that
\begin{equation*}
  G_\eps(u^0_\eps,A^0_\eps) = G_{m}(\hex) > G_\eps(u^\eps_\textit{min}, A^\eps_\textit{min}) \,.
\end{equation*}
Here $(u^\eps_\textit{min}, A^\eps_\textit{min})$ denotes the global minimizer of the energy $G_\eps$ with applied magnetic field $\hex$.  
In this case the energy minimizing configuration has lower energy than the Meissner state, and the dynamic transition to the global minimizer will involve nucleating vortices.

The process by which vortices are nucleated is not yet well understood.
It was shown in \cite{BBC} and  \cite{ChapmanNucleation}  that the Meissner state is linearly stable until the applied magnetic field $\hex$ crosses the second critical field
\[
h_{c_2} \defeq \frac{C }{ \eps},
\]
where $C$ is a constant depending on the domain $\Omega$.
Since $h_{c_2}$ is much larger than $h_{c_1}$, this points at a very significant hysteresis phenomena  and the process by which this dynamically generates vortices is highly nontrivial, see \cite{DuLinHysterisis}.

Along these lines, vortices can also dynamically nucleate as a way of tunneling to lower energy states. 
Due to  topological considerations, vortices should either nucleate at the boundary or nucleate as a dipole in the interior of the domain.  In the first case, let $a_\eps = a_\eps(t)$ be the distance of the center of a vortex from the boundary of $\Omega$.
We know from \cite{SSGamma} that
when $a_\eps(0) > \exp( -|\ln \eps|^{1/2} )$
the evolution of $a_\eps$ is governed by the ODE
\begin{equation} \label{e:VortexMotionInterior}
\pi \dot{a}_\eps = - d\, \lambda \nabla \xi_m (a_\eps) \,,
\end{equation} 
where $d = \pm 1$ is the degree of the vortex.
Hence, any positive vortices move towards the interior to a lower energy state and any negative vortices move to the boundary and become excised (see \cite{SSGamma}).
However, the energy of a vortex at a distance of order~$\exp(-|\ln \eps|^{1/2})$ away from~$\partial \Omega$ is
\begin{equation*}
  G_m(\hex) + \pi \logep + o_\eps(\logep) \gg G_m(\hex)\,.
\end{equation*}
This is an extremely large barrier to overcome and is even more dramatic when a dipole is nucleated in the interior.
In particular once a dipole has separation of at least $\exp\paren{-|\ln \eps|^{1/2}}$, the associated energy is
\begin{equation*}
  G_m(\hex) + 2\pi \logep + o_\eps(\logep) \gg G_m(\hex) \,.
\end{equation*}
Thus the energy barrier to nucleate a vortex at the boundary or via a dipole is much larger than the energy gap between the Meissner state and the configuration with the vortex at the minimizer.  
\medskip

The main purpose of this paper is to better understand how this energy barrier can be overcome through the study of vortices close to the boundary and dipoles with small inter-vortex separation.
In particular, for every $\alpha \in (0, 1)$ we study the dynamics of vortices a distance of $\eps^\alpha$ away from the boundary.
The energy barrier to nucleate such vortices is $\pi (1 -\alpha) \logep$, which is a much smaller energy barrier to overcome.
We show the following results:

\begin{enumerate}
  \item
    We obtain a heuristic ODE governing the motion of these vortices (equation~\eqref{e:ODEa}, below).

  \item
    We rigorously estimate the annihilation times of vortices $O(\eps^\alpha)$ away from the boundary, and show that this agrees with the annihilation times of~\eqref{e:ODEa} (Theorem~\ref{t:annihilation} and Proposition~\ref{p:Tann}, below).

  \item
    We consider a stochastically perturbed version of the heuristic ODE governing vortex motion, and estimate the chance that the vortex nucleates (and thus achieving a lower energy state) before annihilating. This is Theorem~\ref{t:NucProb}, below.
\end{enumerate}
The same analysis can be made for vortex dipoles with inter-vortex separation~$\eps^\alpha$.

A more physically relevant problem is the direct study of a stochastically perturbed version of~\eqref{e:GE1}--\eqref{e:GE4}, without relying on the simplified heuristics.
This is a much harder question requiring a deep understanding of the long time dynamics of the underlying nonlinear stochastic PDE.
The problem is described briefly at the end of Section~\ref{s:results}, below, but its resolution is beyond the scope of the current investigation.


\subsection*{Plan of this paper}
In Section~\ref{s:results} we state the main results of this paper.
In Section~\ref{s:derivation} we formally derive the heuristic ODE~\eqref{e:ODEa} by matching terms of leading order.
In Section~\ref{s:AnnihilationTimes} prove Theorem~\ref{t:annihilation}, rigorously estimate the annihilation times of vortices a distance $O(\eps^\alpha)$ away from the boundary.
Confirming that these annihilation times agrees with that of~\eqref{e:ODEa} is relegated to Appendix~\ref{s:ODEAnnTimes}.
Finally, in Section~\ref{s:stochastic} we prove Theorem~\ref{t:NucProb}, estimating the chance of vortex nucleation.

\subsection*{Acknowledgements}
We thank the anonymous referees for many helpful suggestions and comments.

\section{Main Results.}\label{s:results}


\subsection{Boundary Vortex Dynamics and Annihilation Times}

We begin with a heuristic ODE governing the motion of a vortex close to the boundary of the domain~$\Omega$.
Since the scales we are interested in are very small, we locally flatten the boundary of~$\partial \Omega$ and state the governing equation on the half plane.

\begin{heuristic}
  Let $(0, a_\eps(t))$ be the position of a vortex at time~$t$ in the domain~$\R^2_+ \defeq \R \times \R_+$.
  If $\alpha \in (0,1)$ and $a_\eps(0) = \eps^\alpha$ then, to leading order the motion of the vortex is governed by the ODE
  \begin{equation}\label{e:ODEa}
    \dot a_\eps =  \frac{\lambda \hex }{ \logep} - \frac{1}{\logep a_\eps} \,,
  \end{equation}
  where $\lambda = -2 \partial_y \xi_m(0, 0) > 0$, which corresponds an order parameter of the form \eqref{e:approxuform}.
\end{heuristic}

We provide a formal derivation of~\eqref{e:ODEa} in Section~\ref{s:derivation} by matching terms of leading order.
To obtain a rigorous result supporting~\eqref{e:ODEa} as a model for boundary vortex dynamics,  we show that annihilation time of vortices at a distance of $\eps^\alpha$ from the boundary is $O(\eps^{2\alpha})$
in the full Gor'kov-\'Eliashberg system \eqref{e:GE1}-\eqref{e:GE4} (Theorem~\ref{t:annihilation},  below).
This this agrees with the annihilation time predicted by the ODE~\eqref{e:ODEa} (Proposition~\ref{p:Tann}, below).

In order to state Theorem~\ref{t:annihilation} we need to introduce some notation.
Define the Jacobian $J$ by
\begin{equation*}
  J(w) \defeq \det \nabla w = \frac{1}{2} \curl j(w) \,,
\end{equation*}
where $j(w) \defeq (i w, \nabla w)$.
Recall that if
\begin{equation*}
  w = \paren[\Big]{\prod_{j=1}^n \paren[\Big]{ \frac{ x - a_j}{ \abs{x - a_j}} }^{d_j} }
\end{equation*}
is the order parameter associated with $n$ point vortices located at $a_1$, \dots, $a_n$ with degrees $d_1$, \dots, $d_n \in \set{\pm 1}$ respectively,
then a direct calculation shows
\begin{equation*}
  J\paren[\Big]{\prod_{j=1}^n \paren[\Big]{ \frac{ x - a_j}{ \abs{x - a_j}} }^{d_j} }
  = \pi \sum_{j=1}^n d_j \delta_{a_j} \,.
\end{equation*}
Consequently, $J$ can be used to describe the location of vortices.

More precisely, the measure of vortex separation used throughout this paper is the $(C^{0,\gamma}_0(\Omega))^*$ norm of the differences in the Jacobian of the order parameters.
Here $0 < \gamma \leqslant 1$ if a fixed parameter.
It is convenient to note that if $a, b \in \Omega$ are such that $\min \{\dist(a,\partial \Omega), \dist(b,\partial \Omega) \} \geqslant |a-b|$, then 
\begin{equation*}
  \norm[\big]{ \delta_{a} - \delta_b }_{(C^{0,\gamma}_0(\Omega))^*} = | a - b |^\gamma \,.
\end{equation*}
We now state our first result.

\begin{theorem} \label{t:annihilation}
Let $(u_\eps(t), A_\eps(t), \Phi_\eps(t))$  be a solution to the system \eqref{e:GE1}--\eqref{e:GE4} 
under the Coulomb gauge such that
\begin{equation*}
  C_1 \logep \leqslant \hex \leqslant C_2 \exp\paren[\big]{\logep^{1/2}}\,.
\end{equation*}
for some constants $C_1, C_2 > 0$.
Suppose the initial data $(u_\eps^0, A_\eps^0, \Phi^0_\eps)$ satisfies
 \begin{equation*}
   \abs[\big]{ G_\eps(u_\eps^0, A^0_\eps) - G_{m}(\hex) } \leqslant \pi ( 1 - \alpha) \logep + C
 \end{equation*}
for a constant $C$ and $0 < \alpha < 1$. Moreover, for some~$\gamma \in (0, 1]$ suppose
 either
\begin{equation}\label{e:IDDipole}
  \norm[\Big]{ J(u^0_\eps) - \pi \paren[\big]{  \delta_{a_\eps^{0, +}} - \delta_{a_\eps^{0, -}} } }_{(C_0^{0,\gamma})^*} =o_\eps\paren[\Big]{\frac{\logep }{ \hex }} \text{ and }  \abs[\big]{ a^{0, +}_\eps - a^{0, -}_\eps } =  \eps^\alpha \,,
\end{equation}
 or 
\begin{equation}\label{e:IDBoundary}
  \norm[\big]{ J(u^0_\eps ) -  \pi \delta_{a^0_\eps} }_{(C_0^{0,\gamma})^*}  = o_\eps \paren[\Big]{ \frac{\logep }{ \hex} } \text{ and } \dist(a^0_\eps, \partial \Omega) =  \eps^\alpha \,.
\end{equation}
Then there exists a time $t_\eps \in ( \frac{\eps^{2\alpha}}{2}, \eps^{2\alpha})$ such that
\begin{equation}\label{e:annihilated}
  \norm[\big]{ 1 - \abs{u_\eps(t_\eps)} }_{L^\infty(\Omega)} = o_\eps(1).
\end{equation}
In particular, there are no vortices in the domain at time $t_\eps$.
\end{theorem}
\begin{remark*}
  Although Theorem~\ref{t:annihilation} guarantees that there are no vortices in the domain at time $t_\eps$, we do not know whether the Meissner state persists at later times, unlike in the non-gauged case~\cite{SerfatyCollision2}. 
  This is because of the loss of fine tuned control of the energy decomposition at larger times in the gauged problem.
  We remark, however, that the annihilation time scale in Theorem~\ref{t:annihilation} is of the same order as that in~\cite{SerfatyCollision2} in the non-gauged case.
\end{remark*}
\begin{remark*}
      Initial data satisfying~\eqref{e:IDDipole} can be constructed as follows.
        Let $\rho_\eps$ be the profile of the equivariant vortex as in \eqref{e:approxuform}, and define
	\begin{equation*}
	  u_\eps^0 = \prod_{j=1}^2 \rho_\eps(|x -a^\eps_j|) \paren[\Big]{ \frac{x - a^\eps_j }{ | x - a^\eps_j |} }^{d_j}  e^{i\psi^*_\eps} \,,
	\end{equation*}
	with $d_1 = 1, d_2 = -1$, $a_1^\eps = a_\eps^{0, +}$, $a_2^\eps = a_\eps^{0, -}$.
	We set $A_\eps^0 = \hex \nabla^\perp \xi_m$ (where~$\xi_m$ is given by~\eqref{e:Meissner}) and $\Phi^0_\eps \equiv 0$.
	A short calculation shows that
	$G_\eps (u_\eps^0, A_\eps^0) = G_m + \pi (1 - \alpha)  \logep + o_\eps(\logep)$.
\end{remark*}

For completeness, we also estimate the annihilation time of the ODE~\eqref{e:ODEa},
and show that it is of the same order as the time scales obtained in Theorem~\ref{t:annihilation}, up to a logarithmic factor. 
\begin{proposition}\label{p:Tann}
  Let $\alpha \in (0, 1]$ and suppose $a_\eps$ satisfies the ODE~\eqref{e:ODEa} with initial data $a_\eps(0) = \eps^\alpha$.
  Let $t_\eps$ be the vortex annihilation time (i.e.\ a time such that $a_\eps(t_\eps) = 0$).
  If $\eps^\alpha \hex \abs{\ln \eps} \to 0$ as $\eps \to 0$, then
  \begin{equation}\label{e:Tann}
    \lim_{\eps \to 0}
      \frac{t_\eps}{\eps^{2\alpha} \abs{\ln \eps}} = \frac{1}{2}
  \end{equation}
  for any $\lambda \geqslant 0$.
\end{proposition}

The proof of Theorem~\ref{t:annihilation}
(presented in Section~\ref{s:AnnihilationTimes})
is similar to arguments found in \cite{SerfatyCollision2} in the gauge-free situation.
The proof of Proposition~\ref{p:Tann} follows quickly from well known properties of the Lambert~$W$ function, and is relegated to Appendix~\ref{s:ODEAnnTimes}.

\subsection{Stochastic Models for Driving Dipoles.}

In light of Theorem~\ref{t:annihilation}, one requires an applied field \emph{larger} than~$\exp( |\ln \eps|^{1/2} )$ for vortices to be pulled away from the boundary and nucleate.
This is an extremely large energy barrier to overcome.
Our aim is to introduce a small random perturbation into~\eqref{e:ODEa} to account for thermal fluctuations.
In this scenario, nucleation becomes a rare event, and our aim is to estimate the chance of nucleating.
The general idea of studying such asymptotics was introduced in~\cite{PontryaginAndronovEA33}, and has since been extensively studied by various authors (see for instance~\cite{FreidlinWentzell93,Kifer88}).

Explicitly, the equation we consider is
\begin{equation}\label{e:SDEa}
  da_\eps = -b_\eps(a_\eps) \, dt + \sqrt{2\beta_\eps} \, dW_t,
\end{equation}
where
\begin{equation} \label{e:drift}
  b_\eps(a) \defeq -\frac{\lambda \hex}{\abs{\ln \eps}} + \frac{1}{\abs{\ln \eps} a}
\end{equation}
is the right hand side of~\eqref{e:ODEa}, and $W$ is a standard Brownian motion and $\beta_\eps$ is a scaling parameter depending on $\eps$.
We remark that equation~\eqref{e:SDEa} is similar to the SDE satisfied by a Bessel process of dimension $1 + 1/(\beta_\eps \abs{\ln \eps})$.

The question of how stochastic forcing in the Gor'kov-\'Eliashberg equations can induce stochastic ODE's for the vortex position \eqref{e:SDEa}-\eqref{e:drift} is not well understood.  Some initial forays in this direction 
in the gauge-less case can be found in~\cite{Chugreeva}.

  Once vortices reach a distance of $\exp(- |\ln \eps|^{1/2})$ away from the boundary, we know (see~\cite{SSGamma}, briefly described in Section~\ref{s:Intro}) that they are driven into the interior and move into a stable, lower energy state.
  In our context, if the boundary vortex dynamics is approximated by~\eqref{e:ODEa}, then vortices which are a distance of $O(1 / \hex)$ will be pulled away from the boundary.
  Thus, in order to investigate nucleation, we study the chance that solutions to~\eqref{e:SDEa} starting a distance of~$\eps^\alpha$ away from the boundary, reach a distance $O(1 / \hex)$ \emph{before} annihilating (i.e.\ before $a_\eps = 0$, corresponding to the vortex with center $a_\eps$ reaching the boundary of the domain).

  Precisely, define the stopping time $\tau_\eps$ by
  \begin{equation}\label{e:tau}
    \tau_\eps = \inf\set{ t \:|\: a_\eps(t) \not\in (0, \hat A_\eps) } \,,
    \quad\text{where}\quad
    \hat A \defeq \frac{1}{\lambda \hex}\,.
  \end{equation}
  Since the vortices we consider have radius $\eps^\alpha$, the closest to the boundary they can start is a distance of $\eps^\alpha$ away.
  In this case the chance of nucleating one such vortex is
  \begin{equation*}
    \prob^{\eps^\alpha} \paren[\big]{ a_\eps(\tau_\eps) = \hat A_\eps } \defeq \prob\paren[\big]{ a_\eps(\tau_\eps) = \hat A_\eps \given a_\eps(0) = \eps^\alpha } \,.
  \end{equation*}

  On a bounded domain $\Omega$, we locally flatten the boundary and interpret $a_\eps(t)$ as the distance of the vortex from the boundary.
  In this case it is natural to consider the motion of many vortices simultaneously, near different points on the boundary.
  Since the size of these vortices is $O(\eps^\alpha)$, we can fit $O(\eps^{-\alpha})$ such vortices simultaneously on $\partial\Omega$.
  Assuming the motion of each of these vortices is independent, and governed by~\eqref{e:SDEa}, then the chance that \emph{at least} one of these vortices nucleates is given by
  \begin{equation}\label{e:Pnuc}
    N_\eps \defeq 1 - ( 1 - \prob^{\eps^\alpha}( a_\eps(\tau_\eps) = \hat A_\eps ) )^{\eps^{-\alpha}}.
  \end{equation}
  Under physically relevant assumptions on $\hex$, we show that the nucleation probability $N_\eps$ transitions from $0$ to $1$ at the threshold
  \begin{equation*}
    \beta_\eps \approx \frac{\alpha}{\ln \abs{\ln \hex} } \,.
  \end{equation*}
  This, along with more precise asymptotics, is our next result.
  \begin{theorem}\label{t:NucProb}
    Let $a_\eps$ solve the SDE~\eqref{e:SDEa}--\eqref{e:drift},
    \begin{equation}\label{e:C0}
      c_\eps \defeq \frac{\lambda \hex}{\abs{\ln \eps}}\,,
      \quad
      \hat A \defeq
	\frac{1}{\lambda \hex}
	\,,
    \end{equation}
    and $\tau_\eps$, defined by~\eqref{e:tau}, be the first exit time of $a$ from the interval $(0, \hat A_\eps)$.
    Suppose%
    \footnote{
      Clearly $\hex \approx \abs{\ln \eps}$, and $\hex \approx \exp\paren{|\ln \eps|^{1/2}}$ as in Theorem~\ref{t:annihilation} have the property that as~$\eps \to 0$, both $\hex \to \infty$ and $\eps^s \hex \to 0$ for all $s > 0$.
    }
    that as $\eps \to 0$, we have $\hex \to \infty$ and $\eps^s \hex \to 0$ for any $s > 0$.
    \begin{enumerate}
      \item If
	\begin{equation*}
	  \limsup_{\eps \to 0} \beta_\eps \ln \hex  < \alpha,
	\end{equation*}
	then the nucleation probability $N_\eps \to 0$ as $\eps \to 0$.

      \item On the other hand, if
	\begin{equation*}
	  \lim_{\eps \to 0} \beta_\eps = 0
	  \quad\text{and}\quad
	  \liminf_{\eps \to 0} \beta_\eps \ln \hex > \alpha,
	\end{equation*}
	then the nucleation probability $N_\eps \to 1$ as $\eps \to 0$.
    \end{enumerate}
  \end{theorem}

  We remark that $\hat A_\eps$ in~\eqref{e:C0} is chosen so that $a = \hat A_\eps$ is a stable equilibrium of the ODE~\eqref{e:ODEa}.
  The proof of Theorem~\ref{t:NucProb} is in Section~\ref{s:stochastic}, and also provides asymptotics in the transition regime when
  \begin{equation*}
    \beta_\eps \ln \hex \to \alpha \,,
  \end{equation*}
  In this case limiting value of $N_\eps$ depends on the rate of convergence and is described in Remark~\ref{r:Nprob}, below.

  We conclude this section with the description of an open question that is more physically realistic.
  Consider a stochastically forced version of the full Gor'kov-\'Eliashberg equations~\eqref{e:GE1}--\eqref{e:GE4}, instead of the heuristic ODE~\eqref{e:ODEa}.
  The noise should spontaneously generate vortices, and due to topological constraints these vortices will appear either near the boundary, or as dipoles with small inter vortex separation.
  Using the heuristic ODE~\eqref{e:ODEa} and Theorem~\ref{t:NucProb} we expect that when the noise is strong enough, these vortices (or dipoles) will nucleate providing a mechanism by which the system ``tunnels'' to a lower energy state.
  This leads us to make the following conjecture.

  \begin{conjecture}
    Consider a stochastically forced version of~\eqref{e:GE1}--\eqref{e:GE4}.
    If the forcing is strong enough, the system admits a unique invariant measure for all $\eps > 0$.
    In this case,  the invariant measure converges weakly as~$\eps \to 0$ to a measure supported on the set of all functions that are limits of global minimizers of the Ginzburg-Landau energy functional.
    Depending on the relation between $\hex$ and $h_{c_1}$ such functions correspond to the purely superconducting state, or a nucleated state with finitely many vortices.
  \end{conjecture}

  In light of Theorem~\ref{t:NucProb} one would guess that the above conjecture holds when the variance of the noise is at least~$O(1/\ln \hex )$.
  However, this would require the stochastic forcing to spontaneously nucleate enough dipoles (or vortices near the boundary).
  Moreover, truly nonlinear effects may change the threshold significantly.

  Proving existence (and possibly uniqueness) of the invariant measure is likely to be amenable to currently available techniques.
  Understanding the limiting behaviour of the invariant measure, however, is more delicate.
  In finite dimensions, the small noise limit of the invariant measure of a randomly perturbed potential flow is a sum of delta masses located at the global minima of the potential, with the relative mass at each minima depending on its basin of attraction.
  In infinite dimensions the situation is more complicated as there may be a continuum of global minima, and understanding the limiting behaviour is much more involved.

  The remainder of this paper is devoted to proving the results stated in this section.

\section{Formal Derivation of  Boundary Vortex Dynamics.}\label{s:derivation}

The purpose of this section is to provide a short, heuristic derivation of \eqref{e:ODEa}.
Recall that a standard calculation (see for example \cite{Du}) shows that
\begin{align}
\label{e:energyid1}
\partial_t g_\eps (u_\eps,A_\eps) + \abs{ \partial_{\Phi_\eps} u_\eps }^2 + \abs{E_\eps}^2 & = \dv \paren{ \partial_{\Phi_\eps} u_\eps , \nablaAe u_\eps } + \curl \paren{ E_\eps \paren{ h_\eps - \hex } } \,,
\end{align}
where $g_\eps(u,A)$, defined by
\begin{align*}
g_\eps(u,A) & \defeq \frac{1}{2} \abs{ \nablaA u }^2 + \frac{1}{2} \abs{ h - \hex }^2 + \frac{1}{4\eps^2} \paren[\big]{ 1 - |u|^2 }^2 \,,
\end{align*}
is the energy density associated to $G_\eps(u,A)$.
We will split this energy density into simpler terms.
Following Bethuel et.\ al.~\cite{BBH}, let%
\begin{align*}
\Ee (u) \defeq \int_\Omega e_\eps(u) \, dx \,,
\quad\text{where}\quad
e_\eps(u) & \defeq \frac{1}{2} \abs{ \nabla u }^2  + \frac{1}{4\eps^2} \paren[\big]{ 1 - |u|^2 }^2 \,.
\end{align*}
as introduced and studied by Bethuel et.\ al.~\cite{BBH}.

Suppose now that our domain $\Omega$ is the half ball $\Omega = B_1(0) \cap \R^2_+$,
and consider a vortex located at $(0, a_\eps(t))$ at time $t$, where $a_\eps(0) = a_\eps^0 = \eps^\alpha$.
From~\cite{SerfatyLocal} we know
\begin{align*}
G_\eps(u_\eps, A_\eps) = G_m(\hex) + \Ee(u_\eps) + 2 \pi \hex \int_\Omega \xi_m  J(u_\eps) dx +  \text{ lower order terms} \,,
\end{align*}
where $G_m(\hex)$ is the Meissner energy associated to applied field $\hex$.
We can approximate the energy $\Ee(u_\eps)$ by
\[
  \Ee(u_\eps) = \pi \ln \frac{ a_\eps(t) }{ \eps} +
\text{ lower order terms} \,.
\]
Combining the previous two equations, and using the fact that on small scales $J(u_\eps)$ is concentrated at the site of the vortex (see~\cite{JerrardSoner}), yields
\begin{align*}
G_\eps(u_\eps, A_\eps) = \pi \ln \frac{ a_\eps(t) }{ \eps} + 2 \pi \hex \, \xi_m(0,a_\eps(t)) + G_m(\hex) +
\text{ lower order terms} \,.
\end{align*}

On the other hand one can formally show that
\[
\int_0^t \int_\Omega \paren[\Big]{ \abs{ \partial_{\Phi_\eps} u_\eps }^2 + \abs{E_\eps}^2 } \, dx \, ds =  \int_0^t \pi \ln \frac{a_\eps(s) }{ \eps} \abs{ \dot a_\eps (s) }^2 ds +
\text{ lower order terms}.
\]
Combined with~\eqref{e:enlaw} this yields
\begin{equation} \label{e:ODEenergy}
\begin{split} 
& \brak[\Big]{ \pi  \log \frac{a_\eps(t) }{ \eps}  + 2 \pi \hex\, \xi_m( 0, a_\eps(t) ) } + \int_0^t \pi  \ln \frac{a_\eps(s) }{ \eps} \abs{ \dot a_\eps (s) }^2 ds \\
& \qquad \qquad = \brak[\Big]{ \pi  \log \frac{a_\eps^0 }{ \eps} + 2 \pi \hex\,  \xi_m (a_\eps^0) } +
\text{ lower order terms}.
\end{split}
\end{equation}
Differentiating \eqref{e:ODEenergy} in time and neglecting the lower order terms yields the ODE
\begin{equation*} 
 \dot a_\eps \ln \frac{a_\eps }{ \eps}=  -2 \hex\, \partial_y \xi_m (0, a_\eps)  - \frac{1}{a_\eps} .
\end{equation*}
Using~\eqref{e:Meissner} and the Hopf lemma, we know that the outward normal derivative of $\xi_m$ is strictly positive on the boundary.
Thus, to leading order, we obtain~\eqref{e:ODEa}.
\section{Dipole Annihilation Times.}\label{s:AnnihilationTimes}


The main goal in this section is to prove Theorem~\ref{t:annihilation}.
We do this through the following $\eta$-compactness result.

\begin{proposition}[$\eta$-compactness] \label{p:etacpt}
Fix $C_1, C_2 > 0$ and suppose that
\begin{equation*}
  C_1 \logep \leqslant \hex \leqslant C_2 \exp\paren[\big]{\logep^{1/2}}\,.
\end{equation*}
Let $(u_\eps(t), A_\eps(t), \Phi_\eps(t))$  be a solution to \eqref{e:GE1}-\eqref{e:GE4} 
under a Coulomb gauge that satisfies 
\[
  \abs[\big]{ G_\eps(u_\eps^0, A^0_\eps) - G_{m} } \leqslant \eta \logep + C
\]
for some constants $C>0$ and $\eta \in(0, \pi)$.
If further
\begin{equation} \label{e:JacCond}
  \norm*{ J(u^0_\eps) }_{(C_0^{0,\gamma})^*} =o_\eps\paren[\Big]{ \frac{\logep }{ \hex} }
\end{equation}
for some $0 < \gamma \leqslant 1$, then for any $\delta < { 2 - \frac{2 \eta }{ \pi}}$, there exists a time $ t_\eps \in (\frac{\eps^\delta}{2}, \eps^\delta)$ such that~\eqref{e:annihilated} holds.
\end{proposition}

Momentarily postponing the proof of Proposition~\ref{p:etacpt}, we prove Theorem~\ref{t:annihilation}.

\begin{proof}[Proof of Theorem~\ref{t:annihilation}]
%
We first consider the case where the initial data is a dipole with separation~$\eps^\alpha$ (i.e.\ satisfies~\eqref{e:IDDipole}).
In this case,
\begin{align*}
\| J(u^0_\eps) \|_{(C^{0,\gamma}_0)^*} 
& \leqslant \norm[\Big]{ J(u^0_\eps) - \pi \paren[\big]{ \delta_{a_\eps^{0,+}} - \delta_{a_\eps^{0,-}}} }_{(C^{0,\gamma}_0)^*}
  + \pi \norm[\Big]{  \delta_{a_\eps^{0, +}} - \delta_{a_\eps^{0,-}} }_{(C^{0,\gamma}_0)^*} \\
& = o_\eps \paren[\Big]{ \frac{\logep }{ \hex} } + \pi \abs[\big]{a_\eps^{0,+} - a_\eps^{0,-}}^\gamma  =o_\eps\paren[\Big]{\frac{\logep}{\hex}} \,.
\end{align*}
Thus, Proposition~\ref{p:etacpt} guarantees that for any $\delta \in [0, 2 - 2\eta / \pi)$, there exists a $t_\eps < \eps^\delta$ such that equation~\eqref{e:annihilated} holds.
Since $\eta = \pi \paren{ 1 - \alpha }$, the restriction $\delta < 2 - 2\eta / \pi$ is precisely $\delta < 2 \alpha$.
Taking the infimum of the times $t_\eps$ as $\delta \to 2\alpha$ will guarantee the existence of a time in the interval $[0, \eps^{2\alpha}]$ for which~\eqref{e:annihilated} holds.
This proves Theorem~\ref{t:annihilation} in the case that~\eqref{e:IDDipole} holds.

The proof when the initial data is a vortex located a distance of $\eps^\alpha$ away from the boundary (i.e.\ when~\eqref{e:IDBoundary} is satisfied) is similar.
Indeed, let $\varphi \in C^{0,\gamma}_0$ and $y \in \partial \Omega$ be the point that is closest to the vortex center~$a_\eps^0$, and observe
\begin{equation*}
  \abs[\Big]{ \int_\Omega \varphi \delta_{a_\eps^0} \, dx } = \abs[\big]{  \varphi(a_\eps^0) - \varphi(y) } \leqslant \dist(a_\eps^0, \partial \Omega)^\gamma \| \varphi \|_{C^{0,\gamma}_0} \,.
\end{equation*}
Thus
$\| \delta_{a_\eps^0} \|_{(C_0^{0,\gamma})^*} = O(\eps^{\alpha \gamma})$,
and hence
\begin{equation*}
\| J(u_\eps^0) \|_{(C_0^{0,\gamma})^*} \leqslant \| J(u_\eps^0) -  \delta_{a_\eps^0} \|_{(C_0^{0,\gamma})^*} + \| \delta_{a_\eps^0} \|_{(C_0^{0,\gamma})^*} 
= o_\eps\paren[\Big]{ \frac{\logep}{\hex}} \,.
\end{equation*}
Now Proposition~\ref{p:etacpt}, and the same argument as in the previous case, finishes the proof.
\end{proof}

The remainder of this section is devoted to the proof of Proposition~\ref{p:etacpt}.
We begin by recalling a regularity result from~\cites{DuGlobal,SpirnGE}.

\begin{lemma}[Lemma 3.7 in~\cite{DuGlobal} or Propositions 2.7--2.8 in~\cite{SpirnGE}]\label{l:regularity}
Let $(u_\eps,A_\eps, \Phi_\eps)$ be a solution to \eqref{e:GE1}--\eqref{e:GE4} under the Coulomb gauge with $\norm{u^0_\eps}_{L^\infty} \leqslant 1$, $\norm{ \nabla u^0_\eps }_{L^\infty} \leqslant \frac{C }{ \eps}$ and $G_\eps(u^0_\eps, A_\eps^0) \leqslant C \hex^2$.
Then we have%
\begin{align}
\label{e:uLinfty}
\norm{ u_\eps(t) }_{L^\infty} & \leqslant 1\,, \\
\label{e:graduLinfty}
\norm{ \nabla u_\eps(t) }_{L^\infty} & \leqslant \frac{C }{ \eps} \,,
\end{align}
for all $t \geqslant 0$.
\end{lemma}

\begin{remark*}
  The hypotheses in \cite{DuGlobal} and~\cite{SpirnGE} for~\eqref{e:uLinfty} and~\eqref{e:graduLinfty} respectively, are for smaller energies (when $G_\eps(0) = O(\logep)$) and under the parabolic gauge.
  The proofs, however, can be easily adjusted to the higher energy level $O(\hex^2)$ and the Coulomb gauge as stated in Lemma~\ref{l:regularity} above.
\end{remark*}

The main step in the proof of Proposition~\ref{p:etacpt} is an energy-splitting argument, which we now describe.
Define the free energy $\GF(u, A)$ by
\begin{equation} \label{e:freeGLen}
  \GF(u,A) \defeq \int_\Omega \paren[\Big]{ \frac{1}{2} \abs{ \nablaA u }^2 + \frac{1}{2} \abs{ \curl A }^2 + \frac{1}{4 \eps^2} \paren{ 1 - |u|^2 }^2 } \, dx \,.
\end{equation}
Even though $G_\eps(u_\eps,A_\eps)$ is of order $O(\hex^2)$, we claim that the $\Ee(u_\eps)$ is of order $O(\logep)$ for short time.
This is our next result.

\begin{proposition}\label{p:split}
Suppose that $(u_\eps(t), A_\eps(t), \Phi_\eps(t))$ is a solution to \eqref{e:GE1}-\eqref{e:GE4} in the Coulomb gauge
with $|G_\eps(u^0_\eps, A^0_\eps)  - G_m|  \leqslant \eta \logep$.  If $\norm{ J(u_\eps^0) }_{(C_0^{0,\gamma})^*}= 
o_\eps(\frac{\logep }{ \hex})$ for 
some $0 < \gamma \leqslant 1$, then for all  $0 \leqslant t \ll \paren[\big]{ \frac{\logep }{ \hex^3} }^{ {2 / \gamma }}$ we have%
\begin{equation}  \label{e:Eenupperbd}
  \Ee(u_\eps(t)) \leqslant \eta \logep + o_\eps(\logep).
\end{equation}
\end{proposition}
\begin{remark*}
By the assumptions on $\hex$, we have
\begin{equation*}
  \eps^\beta \ll \paren[\Big]{ \frac{\logep }{ \hex^3} }^{{2 / \gamma }}
\end{equation*}
for any $0 \leqslant \beta \leqslant 1$ and all $\eps \leqslant \eps_0$, independent of $\beta, \gamma$.
\end{remark*}
\begin{remark*}
A dipole separated by $\eps^\alpha$ or a vortex at a distance $\eps^\alpha$ from the boundary satisfy the hypotheses.
\end{remark*}

\begin{proof}[Proof of Proposition~\ref{p:split}]
  \begin{inparaenum}[1.]

  \item
    We first establish some regularity results on solutions of the equation. We will fix a Coulomb Gauge that ensures
\[
 \dv A_\eps = 0 \hbox{ in } \Omega \qquad A_\eps \cdot \nu = 0 \hbox{ on } \partial \Omega \,.
\]
In this gauge, a solution satisfies
\begin{align*}
\partial_t u_\eps - i \Phi_\eps u_\eps & = \nablaAe^2 u_\eps + \frac{1}{\eps^2} u_\eps \paren{ 1 - |u_\eps|^2 } \,,\\
\partial_t A_\eps - \nabla \Phi_\eps & = \Delta A_\eps + j_{A_\eps}(u_\eps) \,,
\end{align*}
in $\Omega$ with boundary conditions
\begin{align*}
\partial_\nu u_\eps & = \nu \cdot A_\eps = \partial_\nu \Phi_\eps = 0\,, \\
{h_\eps }  & = \hex \,,
\end{align*}
on $\partial \Omega$.
Using the boundary conditions \eqref{e:GE3}--\eqref{e:GE4} and \eqref{e:energyid1} we have the energy bound
\begin{equation}
\begin{split} 
 G_\eps(u_\eps(t), A_\eps(t))  + \int_0^t \int_\Omega \abs{ \partial_{\Phi_\eps} u_\eps }^2 + \abs{E_\eps}^2  \, dx \, ds   = G_\eps(u_\eps^0, A_\eps^0)  ,
\end{split}
\end{equation}
and by assumptions on $\hex$,  $G_\eps(t) \leqslant C \hex^2$.  
(We assume here, and subsequently, that $C$ is a constant independent of $\eps$ that may increase from line to line.)
Therefore,
\begin{equation*}
  \norm{ \nablaAe u_\eps }_{L^2} \leqslant C \hex,
\quad
\norm{ E_\eps }_{L^2([0,t];L^2(\Omega))} \leqslant C \hex,
\quad\text{and}\quad
\norm{ A_\eps }_{H^1} \leqslant C \hex,
\end{equation*}
where the last bound on $A_\eps$ follows from the Coulomb gauge and a standard Hodge argument.

\item
  Next we claim that for all $0 \leqslant t \leqslant 1$ and any $0 \leqslant \gamma \leqslant 1$,  
\begin{equation} \label{e:controljactime}
  \norm[\big]{ J(u_\eps(t)) - J(u^0_\eps) }_{(C_0^{0,\gamma})^*} \leqslant {C \hex^2 \paren[\big]{ \max \{ \eps,  t^{\sfrac{1 }{ 2}} \} }^\gamma },
\end{equation}
and if $0 \leqslant t \ll \paren[\big]{ \frac{ \logep }{ \hex^3} }^{{2 / \gamma }}$, then
\begin{equation} \label{e:controljactime2}
  \norm*{ J(u_\eps(t)) - J(u^0_\eps) }_{(C_0^{0,\gamma})^*} = o_\eps\paren[\Big]{ \frac{\logep}{ \hex} }.
\end{equation}
This will enable us to split the full Ginzburg-Landau energy sufficiently well.

We first establish an estimate on the continuity of the Jacobian in certain weak topologies.
Recalling
$J_A(u) = \frac{1}{2} \curl \paren{ j_A(u) + A }$,
a direct calculation shows
\begin{equation} \label{e:gaugejacdt}
\partial_t J_{A_\eps}(u_\eps)   =  \curl \paren{ i \partial_{\Phi_\eps} u_\eps , \nablaAe u_\eps } + \curl \paren[\Big]{ E_\eps \paren[\Big]{ \frac{ 1 - |u_\eps|^2  }{ 2} } } \; .
\end{equation}

Now for any $\varphi \in C^{0,\gamma}_0$ we have
\begin{align*}
  \MoveEqLeft
  \abs[\Big]{ \int_\Omega \paren[\big]{ J_{A_\eps(t)}(u_\eps(t)) - J_{A^0_\eps}(u_\eps^0) } \varphi \, dx } 
   = \abs[\Big]{  \int_0^t \int_\Omega  \varphi(x) \frac{d }{ ds}J_{A_\eps(s)}(u_\eps(s))  \, dx \, ds } \\
&= 2 \abs[\Big]{ \int_0^t \int_\Omega \nabla^\perp \varphi \cdot \paren[\big]{ i \partial_{\Phi_\eps} u_\eps , \nablaAe u_\eps } \, dx \, ds } \\
& \qquad + \abs[\Big]{ \int_0^t \int_\Omega \nabla^\perp \varphi \cdot E_\eps \paren{ 1 - |u_\eps|^2 } \, dx \, ds } \\
&\leqslant 2 \norm{ \nabla \varphi }_{L^\infty} \norm{ \partial_{\Phi_\eps} u_\eps }_{L^2([0,t] \times \Omega)} \norm{ \nablaAe u_\eps }_{L^2([0,t] \times \Omega)}  \\
&\qquad +  \eps  \norm{ \nabla \varphi }_{L^\infty} \norm{ E_\eps }_{L^2([0,t] \times \Omega)} \norm[\Big]{ \frac{ 1 - |u_\eps|^2  }{ \eps} }_{L^2([0,t] \times \Omega)}  \\
& \leqslant 2\sqrt{t}  \norm{ \nabla \varphi }_{L^\infty}  \norm{ \partial_{\Phi_\eps} u_\eps }_{L^2([0,t] \times \Omega)}    \norm{ \nablaAe u_\eps }_{L^\infty([0,t]; L^2 (\Omega))} \\
& \qquad + \eps \sqrt{t}  \norm{ \nabla \varphi }_{L^\infty}  \norm{ E_\eps }_{L^2([0,t] \times \Omega)}    \norm[\Big]{ \frac{ 1 - |u_\eps|^2  }{ \eps}  }_{L^\infty([0,t]; L^2 (\Omega))}\; ,
\end{align*}
and so
\begin{equation*}
  \abs[\Big]{ \int_\Omega \paren[\big]{ J_{A_\eps(t)} (u_\eps(t)) - J_{A_\eps^0}(u_\eps^0) } \varphi \, dx } \leqslant C \norm{ \nabla \varphi }_{L^\infty} \hex^2 \sqrt{t} \,.
\end{equation*}
This implies
\begin{equation} \label{e:W11est}
  \norm*{ J_{A_\eps(t)}(u_\eps(t)) - J_{A_\eps^0}(u^0_\eps) }_{\dot{W}^{-1,1}(\Omega)} \leqslant C  \hex^{2}\sqrt{t} \,.
\end{equation}
However, we note that for any $\varphi \in W_0^{1,\infty}$, 
\begin{align*}
  \abs[\Big]{ \int_\Omega \varphi \paren[\big]{ J_{A_\eps} (u_\eps) - J(u_\eps) } \, dx  } & = \abs[\Big]{ \int_\Omega \nabla^\perp \varphi \cdot \paren[\big]{ j_{A_\eps}(u_\eps) + A_\eps - j(u_\eps) } } \\
& = \abs[\Big]{ \int_\Omega \nabla^\perp \varphi \cdot \paren[\big]{ A_\eps (1 - |u_\eps|^2 ) } } \\
& \leqslant \eps \norm{ \nabla \varphi  }_{L^\infty} \norm{ A_\eps }_{L^2} \norm[\Big]{ \frac{ 1 -|u_\eps|^2 }{ \eps} }_{L^2} \\
& \leqslant C \eps  \norm{ \nabla \varphi  }_{L^\infty} \hex^2 \,,
\end{align*}
where we recall $j(u) \defeq (iu, \nabla u)$.
Consequently
\begin{equation} \label{e:gaugeJacvsjac}
  \norm*{ J_{A_\eps(t)}(u_\eps(t)) - J(u_\eps(t)) }_{\dot{W}^{-1,1}(\Omega)} \leqslant C \eps \hex^{2} \,,
\end{equation} 
for all $t \leqslant \eps^{- \frac{1}{2}}$.
In particular \eqref{e:W11est} and \eqref{e:gaugeJacvsjac} imply
\begin{equation} \label{e:C01est}
  \norm*{ J(u_\eps(t)) - J(u^0_\eps) }_{(C^{0,1}(\Omega))^*} \leqslant \norm*{ J(u_\eps(t)) - J(u^0_\eps) }_{\dot{W}^{-1,1}(\Omega)} \leqslant C \max\{ \eps, \sqrt{t}\} \hex^{2}.
\end{equation}

A similar calculation shows that
\begin{equation} \label{e:C00est}
  \norm*{ J(u_\eps(t)) - J(u^0_\eps) }_{(C^{0}_0(\Omega))^*} \leqslant  C \hex^{2}.
\end{equation}
Using \eqref{e:C01est} and \eqref{e:C00est}, along with interpolation of dual H\"older spaces, (see Jerrard-Soner \cite{JerrardSoner}), 
yields \eqref{e:controljactime}.

\item
  Next, we claim that
\begin{equation} \label{e:energydecomp}
G_\eps(u_\eps(t), A_\eps(t)) = G_{m} + \GF (u_\eps(t), A_\eps'(t)) + o_\eps(\logep) \,,
\end{equation}
where $A_\eps' (t) = A_\eps(t) - \hex \nabla^\perp \xi_m$ and $\GF$ is defined in \eqref{e:freeGLen}.
To see this, we decompose 
\begin{align*}
G_\eps(u_\eps, A_\eps) & = \GF(u_\eps, A'_\eps) + \hex^2 \int_\Omega \abs{ \nabla \xi_m }^2 |u_\eps|^2 + |h_m - 1 |^2 \, dx  \\
& \quad - \hex \int_\Omega \nabla^\perp \xi_m \cdot j_{A'_\eps} (u_\eps) + h'_\eps (h_m - 1) \, dx \,,
\end{align*}
where $h_\eps' = \curl A_\eps'$.
Note $h_m - 1 = \xi_m$, where $h_m \defeq \Delta \xi_m$ is the Meissner magnetic field.
Therefore, 
\begin{align*}
  \MoveEqLeft[5]
  - \hex \int_\Omega \nabla^\perp \xi_m \cdot j_{A'_\eps} (u_\eps) + h'_\eps (h_m - 1) \, dx\\
 &= - \hex \int_\Omega \nabla^\perp \xi_m \cdot j (u_\eps) - \nabla^\perp \xi_m \cdot A'_\eps |u_\eps|^2 + h'_\eps \xi_m \, dx \\
&=  2 \hex \int_\Omega  \xi_m J (u_\eps) \, dx +\hex  \int_\Omega \nabla^\perp \xi_m \cdot A'_\eps (|u_\eps|^2 - 1) \, dx  \,,
\end{align*}
and so 
\begin{align*}
G_\eps(u_\eps, A_\eps) & = G_m + \GF(u_\eps, A'_\eps) +  2 \hex \int_\Omega  \xi_m J (u_\eps) \, dx  \\
& \quad  +\hex  \int_\Omega \nabla^\perp \xi_m \cdot A'_\eps (|u_\eps|^2 - 1) \, dx + \hex^2 \int_\Omega \abs{ \nabla \xi_m }^2 ( |u_\eps|^2 - 1) \, dx \\
& = G_m + \GF(u_\eps, A'_\eps) + 2 \hex \int_\Omega  \xi_m J (u_\eps) \, dx  + o_\eps(\logep) \,.
\end{align*}
Since $\xi_m$ is a smooth function (and hence in any H\"older space), then 
\begin{equation*}
\hex \int_\Omega \xi_m J(u^0_\eps) \, dx = o_\eps(\logep)
\end{equation*}
holds by assumption and \eqref{e:energydecomp} follows from \eqref{e:controljactime}. 

\item
  Finally, we prove that 
\begin{equation} \label{e:bbhenergycon}
 \Ee(u_\eps(t)) \leqslant \eta \logep + o_\eps(\logep)
\end{equation}
 for all $t \leqslant \paren[\big]{ \frac{\logep}{\hex^3}}^{2 /  \gamma}$.

By  \eqref{e:enlaw} and our assumptions
\begin{equation} \label{e:enlaw2}
\begin{split} 
& G_\eps(u_\eps(t), A_\eps(t)) + \int_0^t \int_\Omega \abs{ \partial_{\Phi_\eps} u_\eps }^2 + |\nabla \Phi_\eps |^2 + |\partial_t A_\eps|^2 \, dx \, ds   \leqslant G_m + \eta \logep.
\end{split}
\end{equation}
By steps 2 and 3, we have  
\begin{equation} \label{e:energydecomp1}
  \GF(u_\eps(t), A'_\eps(t))  \leqslant  \eta \logep +  o_\eps(\logep)
\end{equation}
as long as $t \ll \paren[\big]{\frac{ \logep }{ \hex^3}}^{{2 /  \gamma}}$.
Thus, we can continue our decomposition and use 
\begin{align*}
\int_\Omega \abs{ \nablaAep u_\eps }^2 \, dx & = \int_\Omega \abs{ \nabla u_\eps}^2 - 2 \nabla^\perp \xi'_\eps  \cdot j(u_\eps) + |A'_\eps|^2 |u_\eps|^2 \, dx \\
& = \int_\Omega \abs{ \nabla u_\eps}^2 - 2  \xi'_\eps  J(u_\eps) + |A'_\eps|^2 |u_\eps|^2 \, dx \,.
\end{align*}
By \eqref{e:energydecomp1} we have  $\norm{ A'_\eps }_{H^1} \leqslant C {\hex}$ then  $\| \xi'_\eps(t) \|_{C^{0,\gamma}} \leqslant C \hex$
for all $t \geqslant 0$. 
Using $\norm{ J(u^0_\eps) }_{(C^{0,\gamma}_{0})^*}= o_\eps(\frac{\logep}{ \hex})$ and \eqref{e:controljactime}, along with the bound 
on $\xi_\eps'$ yields \eqref{e:bbhenergycon}. 
%
%
%
\end{inparaenum}
\end{proof}

We now state the  $\eta$-compactness result for the gauge-less energy, $\mathcal{E}_\eps$.

\begin{proposition}[Proposition 2.2 in~\cite{SerfatyCollision1}]
\label{p:serfaty}
Suppose $u_\eps$ satisfies the \emph{static} equation
\begin{align*}
\Delta u_\eps + \frac{1}{\eps^2} u_\eps \paren{ 1 - |u_\eps|^2 } & = f_\eps \text{ in } \Omega \,,\\
\partial_\nu u_\eps & = 0  \text{ on } \partial \Omega \,.
\end{align*}
Further, assume $|u_\eps| \leqslant 1$, $\Ee(u_\eps) \leqslant M \logep$, $| \nabla u_\eps| \leqslant \frac{C }{ \eps}$, and $ \norm{ f_\eps }_{L^2} \leqslant \frac{1 }{ \eps^{\beta}}$ for some $\beta < 2$.
Then, after extraction of a subsequence $\eps \to 0$, we can find $R_\eps \to +\infty$ with $R_\eps \leqslant C \logep$ and a family of balls $\cup_{i=1}^n B_i = \cup _{i=1}^n B(a_i, R_\eps \eps)$,
with $a_i$ depending on $\eps$ and $n$ bounded independently of $\eps$, 
such that the following hold.
\begin{enumerate}
\item  As $\eps \to 0$ we have
\begin{equation} \label{e:Linftyoutside}
  \norm{ 1 - |u_\eps| }_{L^\infty (\Omega \backslash \cup_i B(a_i, R_\eps \eps))} \to 0 \,.
\end{equation}
%
%

\item For every $\beta < 1$ and every subset $I$ of $[1,n]$, we have 
\begin{equation} \label{e:degboundSerf}
\beta \pi \sum_{i \in I} d_i^2 \leqslant {\int_{\cup_{i \in I} B(a_i, R_\eps \eps^{1-\beta})} \frac{e_\eps(u_\eps) }{ \logep} dx} + 
C \logep^{\frac{7}{ 2}} \eps^{1 - \beta} \norm{ f_\eps }_{L^2} + o_\eps(1)
\end{equation}

\end{enumerate}
\end{proposition}

We now prove Proposition~\ref{p:etacpt}.

\begin{proof}[Proof of Proposition~\ref{p:etacpt}]
Using the above bounds on the gauged problem above, we follow the template laid out in~\cite{SerfatyCollision2}.

%

\begin{asparaenum}[1.]
\item   We first claim that for any $0 < \delta < 1$, there exists a time $t_\eps \in (\frac{\eps^\delta }{ 2}, \eps^\delta)$ such that 
\begin{equation} \label{e:ptkineticbd}
  \norm[\big]{ \partial_{\Phi_\eps}u_\eps (t_\eps)  }_{L^2}^2 + \norm[\big]{ E_\eps (t_\eps)  }_{L^2}^2  \leqslant C {\eps^{-\delta}} \hex^2.
\end{equation}
(We will later choose $\delta$ to ensure there are no vortices at time $t_\eps$.) 
To prove~\eqref{e:ptkineticbd}, we use the parabolic bound
\begin{align*}
\int_0^t \paren[\Big]{ \norm[\big]{ \partial_{\Phi_\eps} u_\eps (s) }^2_{L^2} + \norm[\big]{ E_\eps (s) }_{L^2}^2 } \, ds
& \leqslant  G_\eps(u_\eps^0, A_\eps^0)   \leqslant  C \hex^2 \,.
\end{align*}
Since we also know
\[
C \hex^2 \geqslant
\int_{\frac{t }{ 2}}^t \norm[\big]{ \partial_{\Phi_\eps} u_\eps (s) }^2_{L^2}  + \norm[\big]{ E_\eps (s)  }_{L^2}^2 \, ds \geqslant \frac{t }{ 2} \inf_{s \in [\frac{t}{ 2}, t]} \paren[\Big]{ \norm[\big]{ \partial_{\Phi_\eps} u_\eps (s) }^2_{L^2} + \norm[\big]{ E_\eps (s)  }_{L^2}^2 } 
\]
then there exists a $t_\eps \in  ( \frac{ \eps^\delta }{ 2} ,\eps^\delta )$ such that  \eqref{e:ptkineticbd} holds. 

\item We claim that this implies that 
\begin{equation}  \label{e:LinfbdA}
  \norm{ A_\eps(t_\eps) }_{L^\infty} \leqslant C {\eps^{-\delta/2}} \hex.
\end{equation}
To see this we note that from \eqref{e:ptkineticbd}, 
\begin{align*}
  \norm{ \nabla h_\eps(t_\eps) }_{L^2} 
& \leqslant \norm{ j_{A_\eps} (u_\eps(t_\eps))  }_{L^2} + \norm{ E_\eps(t_\eps) }_{L^2} \\
& \leqslant C \hex  + C  \eps^{-{\delta / 2}}\hex.
\end{align*}
In particular, $\norm{ h_\eps(t_\eps) }_{H^1} \leqslant C {\eps^{-\delta / 2}} \hex $. Using standard elliptic theory and a Hodge decomposition of $A_\eps$ we find that 
\begin{align}
\nonumber \norm{ A_\eps(t_\eps) }_{L^\infty}  \leqslant C \norm{ A_\eps(t_\eps) }_{H^2} \leqslant C \norm{ \curl (-\Delta_0^{-1}) h _\eps(t_\eps) }_{H^2}  
\leqslant C \norm{ h_\eps(t_\eps) }_{H^1} ,
\end{align}
and \eqref{e:LinfbdA} follows by embedding.

\item
  We now show that the first two steps and the hypotheses allow us to use Proposition~\ref{p:split}.  
From the evolution equation for $u_\eps$ we can write 
\begin{align*}
\Delta u_\eps + \frac{1}{\eps^2} u_\eps \paren{ 1 - |u_\eps|^2 } = f_\eps ,
\end{align*}
where 
\begin{equation} \label{e:fdef}
f_\eps \equiv  2 i A_\eps \cdot \nablaAe u_\eps + |A_\eps|^2 u_\eps - \partial_{\Phi_\eps} u_\eps .
\end{equation}
Each of these terms can be estimated in $L^2$ for some time $t \in (\frac{1}{2}  \eps^\delta,  \eps^{\delta})$.  From 
\eqref{e:LinfbdA} 
and \eqref{e:enlaw}
\begin{equation} \label{e:kinest1}
  \norm{ A_\eps \cdot \nablaAe u_\eps }_{L^2} \leqslant \norm{ A_\eps }_{L^\infty} \norm{ \nablaAe u_\eps }_{L^2} \leqslant C \eps^{-\frac{\delta }{ 2}} \hex
\end{equation}
and
\begin{equation} \label{e:kinest2}
  \norm{ |A_\eps|^2 u_\eps }_{L^2} \leqslant \norm{ A_\eps }^2_{H^1} \leqslant C \hex^2.
\end{equation}
\begin{equation} \label{e:kinest3}
  \norm{ \partial_{\Phi_\eps} u_\eps(t_\eps) }_{L^2} \leqslant C \eps^{-\frac{\delta}{2 }} \hex.
\end{equation}

Then by \eqref{e:kinest1}-\eqref{e:kinest3} for the  $t_\eps \in  (\frac{ \eps^\delta }{ 2} ,\eps^\delta )$ we have 
\begin{equation}  \label{e:fupper}
  \norm{ f_\eps (t_\eps) }_{L^2} \leqslant C \eps^{ - \frac{\delta }{ 2 }}\hex.
\end{equation}

\item
  We can now follow Proposition~2.1 in \cite{SerfatyCollision2}; by the structure result Proposition~\ref{p:serfaty} at 
$t_\eps$ defined above,  there are vortices $\{a_j\}$ of degree $\{d_j\}$ such that for any $\beta < 1$, 
\begin{equation} \label{e:degbd}
\beta \pi \sum_{j} d_j^2 \leqslant \eta + C \logep^{\frac{7}{2}} \eps^{1 - \beta - \frac{\delta }{ 2}} \hex + o_\eps(1).
\end{equation}
We can choose $\beta > \frac{\eta }{ 2 \pi}$ and $\delta < 2 - 2 \beta$, then by \eqref{e:degboundSerf} we have that for all $\eps \leqslant \eps_0$,  $\sum_j d^2_j <2$; consequently,
 $\sum d^2_j \in \{0,1\}$.  However, since the $d_j$'s are nontrivial then either $\norm{ 1 - u_\eps(t_\eps) }_{L^\infty} = o_\eps(1)$ or there exists one vortex $a$ with degree $\pm 1$.   
 
 Suppose now that there exists a single vortex $a$, we again use the argument from \cite{SerfatyCollision2} to get  
 that $E_\eps(u_\eps(t_\eps)) \geqslant \pi \ln \frac{ \ell }{ \eps} + O(1)$, where $\ell = \dist(a, \partial \Omega)$.  
 But upper bound \eqref{e:Eenupperbd} implies that $E_\eps(u_\eps(t_\eps)) \leqslant \eta \logep + o_\eps(\logep)$ which implies that the vortex must satisfy $\ell \leqslant \eps^\mu$ 
for some 
\begin{equation} \label{e:mubound}
\mu \geqslant 1 - \frac{\eta }{ \pi}. 
\end{equation}
By Theorem 2 of \cite{SerfatyCollision1} we have that  
\begin{equation} \label{e:flowerbound}
  \norm{ f_\eps (t_\eps) }_{L^2} \geqslant C \frac{ \eps^{-\mu} }{ \logep },
\end{equation}
for some constant $C$ depending on $\beta$ and $\Omega$.  
 Given \eqref{e:fupper} and \eqref{e:flowerbound} we see that
 \begin{equation}  \label{e:deltamulower}
 2 \mu < \delta. 
 \end{equation}

On the other hand we  can further restrict ${\delta } < 2 -  2\frac{\eta }{ \pi}$, and so with \eqref{e:mubound} we get 
\begin{equation} \label{e:deltamuupper}
{\delta } <  2 \mu,
\end{equation}
and a contradiction between \eqref{e:deltamulower} and \eqref{e:deltamuupper}.  Therefore, 
there 
are no vortices at time $t_\eps$, which implies $\norm{ 1 - \abs{ u_\eps(t_\eps) } }_{L^\infty} = o_\eps(1)$.\qedhere
\end{asparaenum}
\end{proof}

\section{Stochastic Models for Driving Dipoles.}\label{s:stochastic}

We devote this whole section to proving Theorem~\ref{t:NucProb}.
  When $\beta_\eps = 0$ the SDE~\eqref{e:SDEa} reduces to~\eqref{e:ODEa} which has a locally attracting point at $0$.
  Standard large deviation results can be used to estimate the chance that $a$ escapes from $0$, however, these results don't directly apply in this scenario as the initial position, strength of the noise, and the interval length all depend on $\eps$.
  While these obstructions can likely be overcome abstractly, the problem at hand admits an explicit solution and we handle it directly instead.

  We know (see for instance~\cite[\S9]{Oksendal03}) that the function $\varphi_\eps$ defined by
  \begin{equation*}
    \varphi_\eps(x) \defeq \prob^x( a_\eps(\tau_\eps) = \hat A_\eps )\,.
  \end{equation*}
  satisfies the equation
  \begin{equation}\label{e:Vp}
    \beta_\eps \partial_x^2 \varphi_\eps - b_\eps \partial_x \varphi_\eps = 0\,,
  \end{equation}
  with boundary conditions
  \begin{equation*}
    \varphi_\eps(0) = 0\,,
    \quad\text{and}\quad
    \varphi_\eps(\hat A) = 1\,.
  \end{equation*}

  Let $B_\eps = \int b_\eps$ be a primitive of $b_\eps$.
  The solution to~\eqref{e:Vp} is given by
  \begin{equation*}
    \varphi_\eps(z)
      = \frac{ \int_0^{z} e^{B_\eps / \beta_\eps} }
	{ \int_0^{\hat A_\eps} e^{B_\eps / \beta_\eps} }
      = \frac{
	  \int_0^{z} x^{\frac{1}{\beta_\eps \abs{\ln \eps}} }
	      \exp\paren[\Big]{ \frac{ - \lambda \hex x }{\beta_\eps \abs{\ln \eps}} } \, dx
	}
	{
	  \int_0^{\hat A_\eps} x^{\frac{1}{\beta_\eps \abs{\ln \eps}} }
	      \exp\paren[\Big]{ \frac{ - \lambda \hex x }{\beta_\eps \abs{\ln \eps}} } \, dx
	}\,.
  \end{equation*}
  Using~\eqref{e:C0} and making the substitution $y = c_\eps x / \beta_\eps$ yields
  \begin{equation*}
    \prob^z(a_\eps(\tau_\eps) = \hat A_\eps)
    = \varphi_\eps(z)
      = \frac{
	  \int_0^{c_\eps z/\beta_\eps} y^{\frac{1}{\beta_\eps \abs{\ln \eps}} } e^{-y} \, dy
	}
	{
	  \int_0^{1/(\beta_\eps \abs{\ln \eps})}
	      y^{\frac{1}{\beta_\eps \abs{\ln \eps}} } e^{-y} \, dy
	}\,.
  \end{equation*}
  Thus, Theorem~\ref{t:NucProb} now reduces to understanding the asymptotic behaviour of the right hand side as $\eps \to 0$.

  To this end, define
  \begin{equation*}
    m_\eps = \frac{c_\eps \eps^\alpha}{\beta_\eps}
    \quad\text{and}\quad
    n_\eps = \frac{1}{\beta_\eps \abs{\ln\eps}}\,,
  \end{equation*}
  and observe
  \begin{equation}\label{e:PnucGamma}
    \prob^{\eps^\alpha}( a_\eps(\tau_\eps) = \hat A_\eps )
    = \varphi_\eps(\eps^\alpha)
      = \frac{
	  \gamma \paren{n_\eps + 1, m_\eps}
	}{
	  \gamma\paren{n_\eps + 1, n_\eps }
	}\,,
  \end{equation}
  where
  \begin{equation*}
    \gamma( s, x) = \int_0^x t^{s-1} e^{-t} \, dt
  \end{equation*}
  is the incomplete lower gamma function.
  Note
  \begin{equation*}
    \frac{m_\eps}{n_\eps}
      = \lambda \hex \eps^\alpha 
      \xrightarrow{\eps \to 0} 0\,.
  \end{equation*}
  We now split the analysis into cases.

  \begin{proofcases}
    \case{$\beta_\eps \ll 1/\abs{\ln \eps}$}
    In this case $n_\eps \to \infty$.
    Clearly
    \begin{equation}\label{e:LgammaNM}
      \gamma(n_\eps+1, m_\eps) \leqslant \int_0^{m_\eps} t^{n_\eps} \, dt = \frac{m_\eps^{n_\eps+1}}{n_\eps+1}.
    \end{equation}
    To estimate $\gamma(n_\eps+1, n_\eps)$, observe first that for any $x > 1$, $\gamma(s, x)$ is decreasing in $s$ when $s$ is sufficiently large.
    Thus, without we can without loss of generality, assume $n_\eps \in \N$.
    Repeatedly integrating by parts we obtain the identity
    \begin{equation*}
      \gamma(n_\eps + 1, x) = n_\eps! ( 1 - e^{-x} e_{n_\eps}(x) )\,,
      \quad\text{where }
      e_n(x) = \sum_{k=0}^n \frac{x^k}{k!}
    \end{equation*}
    is the truncated exponential.
    Since $e^{-n} e_n(n) \to 1/2$ as $n \to \infty$ we must have
    \begin{equation}\label{e:GammaNp1N}
      \lim_{n \to \infty} \frac{\gamma(n+1, n)}{n!} = \frac{1}{2}\,.
    \end{equation}

    For the numerator $\gamma(n_\eps+1, m_\eps)$, clearly
    \begin{equation*}
      \gamma(n_\eps+1, m_\eps) \leqslant \int_0^{m_\eps} t^{n_\eps} \, dt = \frac{m_\eps^{n_\eps+1}}{n_\eps+1}\,,
    \end{equation*}
    and using~\eqref{e:PnucGamma}, \eqref{e:GammaNp1N} and Sterlings formula we have
    \begin{equation*}
      \lim_{\eps \to 0} \frac{\prob^{\eps^\alpha}(a_\eps(\tau_\eps) = \hat A_\eps)}{\eps^\alpha} = 0\,.
    \end{equation*}

    Finally, to estimate $N_\eps$, equation~\eqref{e:Pnuc} shows
    \begin{equation}\label{e:Nestimate}
      \frac{\varphi_\eps(\eps^{\alpha})}{2\eps^\alpha}
      \leqslant N_\eps
      = 1 - ( 1 - \prob^{\eps^\alpha}( a_\eps(\tau_\eps) = \hat A_\eps ) )^{\eps^{-\alpha}}
      \leqslant
	  \frac{2 \varphi_\eps(\eps^{\alpha})}{\eps^\alpha}\,,
    \end{equation}
    and hence $N_\eps \to 0$ as $\eps \to 0$.

    \case{$\beta_\eps \approx 1/\abs{\ln \eps}$}
    In this case we assume
    \begin{equation*}
      \lim_{\eps \to 0} \frac{1}{\abs{\ln \eps} \beta_\eps} = n_0 > 0\,,
    \end{equation*}
    and so $n_\eps \to n_0$ and $m_\eps \to 0$ as $\eps \to 0$.

    Now the denominator in~\eqref{e:PnucGamma} is $\gamma(n_\eps+1, n_\eps)$ which converges to some constant $c_2 > 0$ as $n_\eps \to n_0$.
    The numerator in~\eqref{e:PnucGamma}, $\gamma(n_\eps+1, m_\eps)$, can again be bounded by~\eqref{e:LgammaNM}.
    This shows
    \begin{equation*}
      \prob^{\eps^\alpha}( a_\eps(\tau_\eps) = \hat A_\eps )
	 \leqslant c \paren{\hex \eps^\alpha}^{1 + n_0} \,,
    \end{equation*}
    for some constant $c > 0$.
    Consequently, using~\eqref{e:Nestimate}, we have $N_\eps \leqslant c \hex^{1 + n_0} \eps^{\alpha n_0} \to 0$ as $\eps \to 0$.

    \case{$\beta_\eps \gg 1/\abs{\ln \eps}$}
    In this case both $m_\eps \to 0$,  $n_\eps \to 0$.
    Using the estimate
    \begin{equation*}
      \frac{e^{-x} x^{s}}{s} \leqslant \gamma(s, x) \leqslant \frac{x^s}{s},
    \end{equation*}
    we see
    \begin{equation}\label{e:Vp1}
      e^{-m_\eps} \paren[\Big]{ \frac{m_\eps}{n_\eps} }^{n_\eps + 1}
      \leqslant \varphi_\eps(\eps^\alpha)
      \leqslant e^{n_\eps} \paren[\Big]{ \frac{m_\eps}{n_\eps} }^{n_\eps + 1}\,.
    \end{equation}

    To compute the limiting behaviour of $N_\eps$, observe
    \begin{equation*}
      \lim_{\eps \to 0} \frac{-\ln (1 - N_\eps)}{\eps^{-\alpha} \varphi_\eps(\eps^\alpha)}
	= \frac{-\ln( 1 - \varphi_\eps(\eps^\alpha) )}{\varphi_\eps( \eps^\alpha )}
	= 1\,,
    \end{equation*}
    where, for simplicity, we assumed the existence of the limit.
    Using~\eqref{e:Vp1},
    \begin{align}
      \nonumber
      \lim_{\eps \to 0} \ln\paren[\Big]{ \frac{\varphi_\eps(\eps^\alpha)}{\eps^\alpha} }
	&= \lim_{\eps \to 0} \ln\paren[\Big]{ \frac{ (m_\eps / n_\eps)^{1 + n_\eps} }{\eps^\alpha} }
	\\
	\nonumber
	&= \lim_{\eps \to 0} \brak[\Big]{ \paren[\Big]{1 + \frac{1}{\beta_\eps \abs{\ln \eps}} } \ln\paren[\Big]{ c_\eps \abs{\ln \eps} } - \frac{\alpha}{\beta_\eps} }
	\\
	\label{e:Nprob}
	&= \lim_{\eps \to 0} \brak[\Big]{
	 \frac{1}{\beta_\eps}\paren[\Big]{
	   \frac{\ln\paren[\big]{ \lambda \hex  }}{\abs{\ln \eps}} - \alpha}
	   + \ln \paren{\lambda \hex}} \,,
    \end{align}
    provided the limits exists.
    From this it follows that if $\liminf \beta_\eps \ln \hex > \alpha$, then the limit above is $+\infty$, and hence $N_\eps \to 1$ as $\eps \to 0$.
    On the other hand if $\limsup \beta_\eps \ln \abs{\ln \eps} < \alpha$, then the limit above is $-\infty$, and hence $N_\eps \to 0$ as $\eps \to 0$.
  \end{proofcases}
  This concludes the proof of Theorem~\ref{t:NucProb}.
\begin{remark}\label{r:Nprob}
  In the transition regime (when $\beta_\eps \ln \abs{\ln \eps} \to \alpha$), we note that $N_\eps$ can be estimated from above and below using~\eqref{e:Nprob}.
  The bounds obtained, however, depend on the rate at which $\beta_\eps \ln \hex$ converges to $\alpha$.
\end{remark}

  \appendix
  \section{Annihilation Times of the Heuristic ODE.}\label{s:ODEAnnTimes}
  We devote this appendix to proving Proposition~\ref{p:Tann}, estimating the annihilation times of the heuristic equation~\eqref{e:ODEa}.
    A direct calculation shows that when $\lambda > 0$ the solution to~\eqref{e:ODEa} is given by
    \begin{equation*}
      a_\eps(t) = \frac{1}{\lambda \hex}
	\brak[\Big]{ 1 + W_0( - C \exp\paren[\Big]{ \frac{\lambda^2 \hex^2 t}{\abs{\ln \eps}} }},
    \end{equation*}
    for some constant $C$.
    Here $W_0$ is the principal branch of the Lambert $W$ function.
    We recall (see~\cite{CorlessGonnetEA96}, or Section~4.13 in~\cite{Olver:2010:NHMF}) that $W_0$ is defined by the functional relation
    \begin{equation*}
      W_0(z e^{z}) = z
    \end{equation*}
    when $z \geqslant -1$.

    Using the initial data $a_\eps(0) = \eps^\alpha$ we find
    \begin{equation*}
      C = (1 - \lambda \hex \eps^\alpha) 
	    e^{- \paren{ 1 - \lambda \hex \eps^\alpha } }.
    \end{equation*}
    Annihilation occurs when $W_0 = -1$ which is precisely when
    \begin{equation*}
      C \exp\paren[\Big]{ \frac{\lambda^2 \hex^2 t_\eps}{\abs{\ln \eps}} }
	= \frac{1}{e}.
    \end{equation*}
    Substituting $C$ above gives
    \begin{equation*}
      \exp\paren[\Big]{ \lambda \hex \eps^\alpha + \frac{\lambda^2 \hex^2 t_\eps}{\abs{\ln \eps}} } = \frac{1}{1 - \lambda \hex \eps^\alpha},
    \end{equation*}
    and hence
    \begin{equation*}
      t_\eps
	= \frac{\abs{\ln\eps}}{\lambda^2 \hex^2}
	  \paren[\Big]{ \abs{ \ln\paren{1 - \lambda \hex \eps^\alpha} } - \lambda \hex \eps^\alpha }
	= \frac{ \eps^{2\alpha} \abs{\ln \eps} }{2} + 
	  \frac{1}{3} \lambda \hex \eps^{3\alpha} \abs{\ln\eps}
	  + \cdots\,.
    \end{equation*}
    Since $\eps^\alpha \hex \abs{\ln \eps} \to 0$ by assumption, dividing both sides by $\eps^{2\alpha} \abs{\ln \eps}$, equation~\eqref{e:Tann} follows.

    It remains to prove~\eqref{e:Tann} when $\lambda = 0$.
    In this case, the exact solution to~\eqref{e:ODEa} is given by
    \begin{equation*}
      a_\eps(t) = \paren[\Big]{ \eps^{2 \alpha} - \frac{2t}{\abs{\ln \eps}} }^{1/2},
    \end{equation*}
    and hence
    \begin{equation}\label{e:TepLambdaEq0}
      t_\eps = \frac{ \eps^{2\alpha} \abs{\ln \eps}  }{2},
    \end{equation}
    for any $\eps > 0$.
    This concludes the proof of Proposition~\ref{p:Tann}.

 \bibliographystyle{habbrv}
 \bibliography{bnucbib,refs}
\end{document}